\documentclass[12pt,a4paper,twoside]{article}
\usepackage{amsmath}
\usepackage{amssymb}
\usepackage{graphicx}
\usepackage{amsfonts}
\usepackage{amsthm}
\usepackage[ruled]{algorithm2e}
\usepackage{longtable}
\usepackage{multirow}
\usepackage{mdframed}
\usepackage{subcaption}
\usepackage{mathtools}
\usepackage{abstract}
\usepackage{cite}

\theoremstyle{definition}
\newtheorem{definition}{Definition}[section]

\theoremstyle{remark}
\newtheorem*{remark}{Remark}

\theoremstyle{theorem}
\newtheorem{theorem}[definition]{Theorem}

\theoremstyle{proposition}
\newtheorem{proposition}[definition]{Proposition}

\theoremstyle{definition}
\newtheorem{example}[definition]{Example}

\theoremstyle{corollary}
\newtheorem{corollary}[definition]{Corollary}

\theoremstyle{lemma}
\newtheorem{lemma}[definition]{Lemma}

\theoremstyle{definition}

\theoremstyle{definition}

\title{M-Representation of Polytopes}
\author{Sebastian Sigl and Matthias Althoff}
\date{\vspace{-5ex}}

\begin{document}

\maketitle

%\section*{Abstract}
%\begin{abstract}
    
%\end{abstract}

$\textbf{Abstract}$ We introduce the M-representation of polytopes, which makes it possible to compute linear transformations, convex hulls, and Minkowski sums with linear complexity in the dimension of the polytopes.
When the polytope is a convex hull of a zonotope and a polytope, the representation size can be smaller than any of the known representations (V-representation, H-representation, and Z-representation).
%represent polytopes by the same number of vectors as the minimum number of vectors required for the V-representation and an additional matrix of exponents. 
%For zonotopes we need less vectors for the representation than for a V-representation.
%or the Z-representation. 
%By using the M-representation, one can 
%As a side effect, we obtain a more efficient algorithm for the Minkowski sum of polytopes in V-representation. 
We also provide a variant of the M-representation: The chain representation is more compact and we can directly use it to compute linear transformations and convex hulls -- for all other operations on the chain representation, one requires a conversion to the M-representation.

%We introduce a novel representation for convex, bounded polytopes which we call M-representation. Polytopes can be represented by at most the number of vectors spanning the V- and Z-representation while also speeding up the computation of linear transformations, convex hulls and Minkowski sums. On top of that, we introduce a new, more efficient algorithm for the computation of Minkowski sums of polytopes in V-representation. Furthermore, we introduce a notation that reduces the representation size in M-representation for Minkowski sums and convex hulls.
%In the second part of this paper, we provide a lower bound for the convex hull algorithm for the computation of a Z-representation from the vertices of a polytope. Furthermore, we give a strategy on how to achieve this lower bound. Apart from that, we adapt this algorithm in a way that decreases the number of generators and we show some applications of our results for convex hulls and Minkowski sums of polytopes.

\section{Introduction}

%Polytopes have a variety of applications. For instance in algebra to compute Gröbner bases \cite{Gritzmann}, in neural networks verification to work with tight approximations of input and output sets \cite{Müller} or in twistor theory, a subpart of theoretical physics, where a polytope called amplituhedron is used \cite{Arkani}. 
The two main representations for convex, bounded polytopes are the well known V-representation and H-representation \cite{Gruenbaum, Ziegler}. The first one represents a polytope by its vertices and the second one uses halfspaces. Recently, the novel Z-representation was introduced in \cite{Kochdumper}, which uses generators multiplied by monomials. The Z-representation is a special case of polynomial zonotopes \cite{Althoff}, which can represent non-convex sets.\\

The Z-representation overcomes several shortcomings of the conventional representations, out of which we provide a few examples: In case the matrix $M$ of a linear transformation of an H-representation is not invertible, the computational complexity of this transformation is exponential in the dimension $d$ \cite{Jones}. The complexity of calculating the Minkowski sum of two polytopes in H-representation is also exponential in $d$ \cite{Tiwary} and the computation of the convex hull of two H-representations is NP-hard \cite{Tiwary}. 
While the linear transformation of the V-representation is trivial, the computation of the Minkowski sum \cite{Gritzmann} and the convex hull \cite{Tiwary} of two polytopes in V-representation are exponential in $d$.
Contrary, the Z-representation has only a polynomial complexity for linear transformations, Minkowski sums and convex hulls with respect to the dimension $d$ \cite{Kochdumper}.\\

Let us have a look at related representation types, which are also surveyed in \cite{Althoffset}. The complexities are described in terms of the number of respective generators of these methods if not stated otherwise. Another representation for polytopes are zonotope bundles. This method presents polytopes as the intersection of a finite number of zonotopes \cite{Althoffzb}. An advantage of this method is that the intersection of two zonotope bundles can be found trivially, but neither the Minkowski sum, nor the convex hull have a closed-form expression \cite{Althoffset}.
Polynomial zonotopes \cite{Althoff} have a polynomial complexity %in the number of generators
for the Minkowski sum and the convex hull, but are not closed under intersections \cite{Althoffset}. Besides polynomial zonotopes there also exist constrained zonotopes \cite{Scott} and constrained polynomial zonotopes \cite{Kochdumpercpz}. These have additional constraints on the factors appearing in the definition of a (polynomial) zonotope. The Minkowski sum can be computed in linear complexity %in the number of generators 
for constrained (polynomial) zonotopes. While convex hulls of constrained polynomial zonotopes can be computed with polynomial complexity \cite{Althoffset}, convex hulls of constrained zonotopes can be computed linearly in the number of generators and constraints on the zonotopes \cite{Koeln}. Another representation are support functions \cite{Ghosh}. Introduced by Minkowski it makes use of the supremum of an inner product. The Minkowski sum and the convex hull can be computed linearly%in the number of directions
, but the computation of intersections of support functions is not solved yet \cite{Althoffset}. Spectrahedra are another way to represent convex sets. A spectrahedron is defined by the positive semi-definite values of a Hermitian linear matrix polynomial. Their representation by matrix polynomials makes them useful in linear programming \cite{Netzer}.\\

As the main result of this paper, we introduce the M-representation for polytopes. Its form is similar to the Z-representation, but we constrain the factors to positive values. This rather subtle change has significant implications and improves many characteristics of the Z-representation: So far, we can only find a Z-representation whose number of generators is quadratic in the number of vertices $n$. The M-representation on the other hand only needs at most as many basis vectors as the V-representation and an additional matrix of exponents. This matrix can be saved efficiently such that it vanishes in the complexity of the representation size of the basis vectors.
Furthermore, we introduce a strategy that uses zonotopes in order to decrease the number of basis vectors even further. Given $n$ vertices in $\mathbb{R}^d$ an M-representation can be computed in $\mathcal{O}(dn)$.\\
%As an auxiliary result, we obtain an algorithm for the computation of the Minkowski sum of polytopes in V-representation for which the number of arithmetic operations is linear in $n$ in contrast to the algorithms that are so far only polynomial in $n$.\\

This paper is organized as follows: In Sec. \ref{preliminaries} we present preliminaries and continue in Sec. \ref{secsrep} by introducing the M-representation. We prove the complexities of different operations in Sec. \ref{operationssrep} and present an algorithm to reduce the number of basis vectors in Sec. \ref{algreducedsrep}. In Sec. \ref{repsize} we propose a variant that reduces the complexity of computations convex hulls.

\section{Preliminaries} \label{preliminaries}

\subsection{Notations}

In the remainder of this paper, we will use the following notations: $[n] = \{1, 2, \dots, n\}$ for $n \in \mathbb{N}$, the symbols $\mathbb{O}$ and $\mathbb{I}$ refer to the matrices filled with zeros and ones with proper dimensions. $I_n$ refers to the identity matrix in $\mathbb{R}^{n \times n}$, $L_n$ to the lower triangular matrix filled with ones in $\mathbb{R}^{n \times n}$ and $[ \textrm{ }]$ denotes the empty matrix. Given a matrix $M \in \mathbb{R}^{d \times d}$, $M_{(i,j)}$ represents the $j$-th entry of matrix row $i$, and $M_{(\cdot,j)}$ the $j$-th column. Furthermore, we will denote a set of the form $\{S(\alpha) \textrm{ } | \textrm{ } \alpha \in [0, 1]^p\}$ as $\{S(\alpha)\}_{\alpha}$.

\subsection{Definitions}

Now we provide some definitions that are important for the rest of the paper. In order to make the paper coherent, we will use the term \textit{polytope} instead of \textit{convex bounded polytope}. Let us first define the V-representation, and the H-representation of polytopes.

\begin{definition}[V-representation]
    Let $v_1, v_2, \dots, v_n \in \mathbb{R}^d$ be the vertices of a polytope $\mathcal{P}$. Then we can define the vertex representation as 
    \begin{equation}
        \mathcal{P} = \Big\{ \sum_{i=1}^n \alpha_i v_i \textrm{ } \Big| \textrm{ } \sum_{i=1}^n \alpha_i = 1, \alpha_i \geq 0 \Big\}.
    \end{equation}
    %We call these vertices \textit{generators}. 
    This representation therefore uses $n$ vectors.
\end{definition}

\begin{definition}[H-representation]
    Let $G \in \mathbb{R}^{h \times d}$ be a matrix and $b \in \mathbb{R}^h$ a vector. The halfspace representation of a polytope $\mathcal{P} \subseteq \mathbb{R}^d$ is
    \begin{equation}
        \mathcal{P}=\{x \in \mathbb{R}^d \textrm{ } | \textrm{ } Gx \leq b \}.
    \end{equation}
    This representation uses $h$ halfspaces.
\end{definition}

%Before we can define the Z-representation, we need further definitions. Since the Z-representation of a polytope has the form of a polynomial zonotope, let us define this term next.

% \begin{definition}[Polynomial Zonotope, \cite{Althoff}]
% Let $c \in \mathbb{R}^d$ be a center point, $G \in \mathbb{R}^{d \times h}$ a generator matrix and $\mathcal{E} \in \mathbb{N}^{p \times h}$ an exponent matrix, then a polynomial zonotope $\mathcal{PZ} \subseteq \mathbb{R}^d$ is defined as

% \begin{equation}
%     \mathcal{PZ} = \{c + \sum_{i=1}^h(\prod_{k=1}^p\alpha_k^{\mathcal{E}_{(k,i)}})G_{(\cdot,i)} \textrm{ } | \textrm{ } \alpha_k \in [-1, 1] \textrm{ } \forall k \in [p]\}.
% \end{equation}

% These variables $\alpha_k$ are called factors and the vectors $G_{(\cdot,i)}$ are generators. $p$ is the number of factors $\alpha_k$ and the number of generators is denoted by $h$. A polynomial zonotope is called \textit{regular} if 

% \begin{equation}
%     \forall j,k \in [h]: \textrm{ } j \neq k \Rightarrow \mathcal{E}_{(j, \cdot)} \neq \mathcal{E}_{(k, \cdot)}.
% \end{equation}

% \end{definition}

%In the following, we will refer to regular polynomial zonotopes if we speak of polynomial zonotopes.\\

Let us now define the Z-representation of a polytope, which is a special case of a polynomial zonotope.

\begin{definition}[Z-representation] \label{defzrep}
Let $c \in \mathbb{R}^d$ be a center point, $G \in \mathbb{R}^{d \times h}$ a generator matrix, and $\mathcal{E} \in \{0, 1\}^{p \times h}$ an exponent matrix, then the Z-representation of a polytope $\mathcal{P} \subseteq \mathbb{R}^d$ is defined as

\begin{equation}
    \mathcal{P} = \Big\{c + \sum_{i=1}^h(\prod_{k=1}^p\alpha_k^{\mathcal{E}_{(k,i)}})G_{(\cdot,i)} \textrm{ } \Big| \textrm{ } \alpha \in [-1, 1]^p \Big\}
\end{equation}

and we write

\begin{equation}
    \mathcal{P} = \langle c, G, \mathcal{E} \rangle_Z.
\end{equation}

The Z-representation of a single point $c$ therefore can be expressed by $\langle c, [ \textrm{ }], [ \textrm{ }] \rangle_Z$. This representation uses $h$ generators.
\end{definition}

The Z-representation is not unique. For example, the polytope in Fig. \ref{exzono} can be represented by the following two sets $\mathcal{P}_1$ and $\mathcal{P}_2$ with $\mathcal{P}_1 = \mathcal{P}_2$.

\begin{equation} \label{exzonorep}
    \begin{split}
        \mathcal{P}_1 & = \Big\langle \begin{pmatrix}
        0\\
        0
        \end{pmatrix}, \Big[\begin{pmatrix}
        1\\
        0
        \end{pmatrix}, \begin{pmatrix}
        -1\\
        -1
        \end{pmatrix}\Big], \begin{pmatrix}
        1 0\\
        0 1
        \end{pmatrix} \Big\rangle_Z\\
        \mathcal{P}_2 & = \Big\langle \begin{pmatrix}
        0\\
        0
        \end{pmatrix}, \Big[\begin{pmatrix}
        -0.5\\
        -0.5
        \end{pmatrix}, \begin{pmatrix}
        -0.5\\
        -0.5
        \end{pmatrix}, \begin{pmatrix}
        0.5\\
        0
        \end{pmatrix}, \begin{pmatrix}
        -0.5\\
        -0.5
        \end{pmatrix}, \begin{pmatrix}
        0.5\\
        0.5
        \end{pmatrix}\Big], 
        \begin{pmatrix}
        1 0 0 1 0\\
        0 1 0 0 1\\
        0 0 1 1 1
        \end{pmatrix}
         \Big\rangle_Z
    \end{split}
\end{equation}

%$\mathcal{P}_2$ is a representation with the least amount of generators the convex hull algorithm introduced in \cite{Kochdumper} could achieve.

\begin{figure}[h]
 \centering
 \includegraphics[width=6cm, height=6cm]{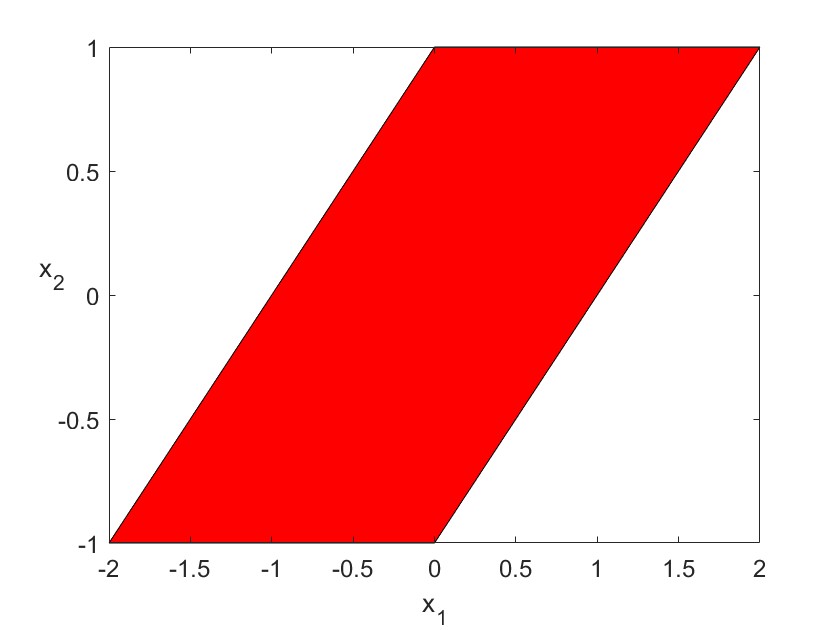}
 \caption{Plot of $\mathcal{P}_1 = \mathcal{P}_2$}
 \label{exzono}
\end{figure}

%EXAMPLE for two Z-representations of the same polytope (parallelogram).

Let us now have a look at zonotopes. The set of zonotopes is a subset of the set of polytopes, we can therefore express zonotopes as a special case of polytopes in Z-representation:

\begin{definition}[Zonotope] \label{zonotope}
A \textit{zonotope} is a polytope in Z-representation with $\mathcal{E} = I_h$. 
\end{definition}

% \begin{proposition}[Linear transformation, \cite{Kochdumper}]
% Let $\mathcal{P}$ be a polytope in $\mathbb{R}^d$. A linear transformation by $M \in \mathbb{R}^{m \times d}$ can be computed as 
% \begin{equation}
%     M \mathcal{P} = \langle Mc, MG, \mathcal{E} \rangle_Z.
% \end{equation}
% \end{proposition}

\begin{definition}[Convex hull, Minkowski sum]
Let $\mathcal{P}_1$ and $\mathcal{P}_2$ be two convex sets. Then we define the convex hull as

\begin{equation} \label{convzreptheory}
    conv(\mathcal{P}_1, \mathcal{P}_2) = \Big\{ \frac{1 + \alpha}{2} p_1 + \frac{1 - \alpha}{2} p_2 \textrm{ } \Big| \textrm{ } p_1 \in \mathcal{P}_1, p_2 \in \mathcal{P}_2, \alpha \in [-1, 1]\Big\}
\end{equation}

and the Minkowski sum as

\begin{equation}
    \mathcal{P}_1 \oplus \mathcal{P}_2 = \Big\{p_1 + p_2 \textrm{ } \Big| \textrm{ } p_1 \in \mathcal{P}_1, \textrm{ } p_2 \in \mathcal{P}_2 \Big\} .
\end{equation}
\end{definition}

The convex hull and the Minkowski sum of two polytopes can be computed in the following way:

\begin{proposition}[Convex hull, Minkowski sum, Linear transformation, \cite{Kochdumper}] \label{propconvmink}
Let $\mathcal{P}_1$ and $\mathcal{P}_2$ be two polytopes in Z-representation with $\mathcal{P}_i= \langle c_i, G_i, \mathcal{E}_i \rangle_Z$, $c_i \in \mathbb{R}^d$, $G_i \in \mathbb{R}^{d \times h_i}$ and $\mathcal{E}_i \in \{0, 1\}^{p_i \times h_i}$. Then their convex hull and Minkowski sum can be computed by

\begin{equation}
    \begin{split}
        conv(\mathcal{P}_1, \mathcal{P}_2) & = \Big\langle \frac{1}{2} (c_1 + c_2), \frac{1}{2} [(c_1 - c_2), G_1, G_1, G_2, -G_2], \hat{\mathcal{E}} \Big\rangle_Z\\
        \mathcal{P}_1 \oplus \mathcal{P}_2 & = \Big\langle c_1 + c_2, [G_1, G_2], \bar{\mathcal{E}} 
        \Big\rangle_Z
    \end{split}
\end{equation}

with
\begin{equation}
    \begin{split}
        \hat{\mathcal{E}} & = \begin{pmatrix}
        \mathbb{O} & \mathcal{E}_1 & \mathcal{E}_1 & \mathbb{O} & \mathbb{O}\\
        \mathbb{O} & \mathbb{O} & \mathbb{O} & \mathcal{E}_2 & \mathcal{E}_2\\
        1 & \mathbb{O} & \mathbb{I} & \mathbb{O} & \mathbb{I}
        \end{pmatrix}\\
        \bar{\mathcal{E}} & = \begin{pmatrix}
        \mathcal{E}_1 & \mathbb{O}\\
        \mathbb{O} & \mathcal{E}_2
        \end{pmatrix}\\
        p & = p_1 + p_2 + 1.
    \end{split}
\end{equation}

For the convex hull we have $h = 2h_1 + 2h_2 + 1$ generators and for the Minkowski sum $h = h_1 + h_2$ generators.\\
A linear transformation by $M \in \mathbb{R}^{m \times d}$ can be computed by 
\begin{equation}
    M \mathcal{P}_1 = \langle Mc_1, MG_1, \mathcal{E}_1 \rangle_Z.
\end{equation}
\end{proposition}

For our M-representation we need another definition:

\begin{definition}[Multilinear map]
    A multivariate map $f: \mathbb{R}^h \to \mathbb{R}^d$ is called \textit{multilinear} if it is linear in every variable.
\end{definition}

In the next section, we introduce our novel M-representation.

\section{M-Representation} \label{secsrep}

By limiting the intervals of the factors of the Z-representation to $[0, 1]$ we obtain the M-representation. This subtle change has far-reaching consequences and combines the computational advantages of the Z-representation with the low representation size of the V-representation.

\begin{definition}[M-representation] \label{defsrep}
Let $s \in \mathbb{R}^d$ be a starting point, $B \in \mathbb{R}^{d \times h}$ a matrix of basis vectors, and $\mathcal{E} \in \{0, 1\}^{p \times h}$ a matrix of exponents, then the \textit{multilinear vertex representation (M-representation)} of a polytope $\mathcal{P} \subseteq \mathbb{R}^d$ is defined as

\begin{equation}
    \mathcal{P} = \Big\{s + \sum_{i=1}^h(\prod_{k=1}^p\alpha_k^{\mathcal{E}_{(k,i)}})B_{(\cdot,i)} \textrm{ } \Big| \textrm{ } \alpha \in [0, 1]^p \Big\}
\end{equation}

and we write

\begin{equation}
    \mathcal{P} = \langle s, B, \mathcal{E} \rangle_M.
\end{equation}

The M-representation of a single point $s$ therefore can be expressed by $\langle s, [ \textrm{ }], [ \textrm{ }] \rangle_M$.
\end{definition}

%Restricting the range of the factors $\alpha_k$ compared to the Z-representation makes it possible to represent the convex hull more compactly compared to \eqref{convzreptheory}, which uses $h = 2h_1 + 2h_2 + 1$ with $h_1$ and $h_2$ the number of basis vectors of $\mathcal{P}_1$ and $\mathcal{P}_2$, respectively:

%\begin{equation}
%    conv(\mathcal{P}_1, \mathcal{P}_2) = \{ \frac{1 + \alpha}{2} \cdot p_1 + \frac{1 - \alpha}{2} \cdot p_2 \textrm{ }|\textrm{ } p_1 \in \mathcal{P}_1, p_2 \in \mathcal{P}_2, \alpha \in [-1, 1]\}
%\end{equation}

%As proven in \cite{Kochdumper}, the number of generators of $conv(\mathcal{P}_1, \mathcal{P}_2)$ can be calculated as $h = 2h_1 + 2h_2 + 1$ with $h_1$ and $h_2$ the number of generators of $\mathcal{P}_1$ and $\mathcal{P}_2$, respectively.\\

%If we instead use the following definition of the convex hull, we get a different number of generators:

% \begin{equation}
%     conv(\mathcal{P}_1, \mathcal{P}_2) = \{ \alpha \cdot p_1 + (1 - \alpha) \cdot p_2 \textrm{ }|\textrm{ } p_1 \in \mathcal{P}_1, \textrm{ } p_2 \in \mathcal{P}_2, \textrm{ } \alpha \in [0, 1]\}
% \end{equation}

% If we sum up the basis vectors with same coefficient in $\alpha$, we get $h = h_1 + 2h_2 + 1$. A proof of this claim can be found in the result of Proposition \ref{srepconv}.

 %It is clear that we can minimize $h(n)$ by setting $n_1 = n - 1$ and $n_2 = 1$. 
%We will use this in the proof of the following theorem.

Now we introduce a theorem that provides a strategy to obtain an M-representation from a set of vertices. Furthermore, we prove how many basis vectors are at most required to represent a general polytope.

\begin{theorem}\label{thmsrep}
Let $v_1, v_2, \dots, v_n \in \mathbb{R}^d$ be the vertices of a polytope $\mathcal{P}$. 
\begin{itemize}
    \item[1.] We can express $\mathcal{P}$ in M-representation as 
\begin{equation}
    \mathcal{P} = \langle v_n, [v_1 - v_2, v_2 - v_3, \dots, v_{n-1} - v_n], L_{n-1}\rangle_M. %\mathcal{E} \rangle_M,
\end{equation}

%where $\mathcal{E} \in \mathbb{R}^{n-1 \times n-1}$ is a lower triangular matrix filled with ones. 
This representation can be obtained with complexity $\mathcal{O}(nd)$ and has a representation size in $\mathcal{O}(nd)$.
\item[2.] This representation has $h = n - 1$ basis vectors.
\end{itemize}
\end{theorem}

We call this form the \textit{chain form}.

\begin{proof}
We prove the statements above by induction.\\

\textit{Induction start}: The M-representation of a polytope with a single vertex $v$ is $\mathcal{P} = \langle v, [ \textrm{ }], [ \textrm{ }] \rangle_M = v$ with zero basis vectors. We can compute the convex hull of two vertices $v_1, v_2$ as 
\begin{equation}
    conv(v_1, v_2) = \{\alpha v_1 + (1 - \alpha) v_2\}_{\alpha} = \{v_2 + \alpha (v_1 - v_2)\}_{\alpha}
\end{equation}
with one basis vector $v_1 - v_2$. Therefore, $h(1) = 0$ and $h(2) = 1$. For $n \leq 2$, these representations have the form described in the theorem.\\

%\textit{Claim of induction}: The M-representation of a polytope with $n$ vertices is $\mathcal{P} = \langle v_n, [v_1 - v_2, v_2 - v_3, \dots, v_{n-1} - v_n], \mathcal{E} \rangle_M$ with $h(n) = n - 1$.\\

\textit{Induction step}: Let $\mathcal{P}_1$ and $\mathcal{P}_2$ be two polytopes in M-representation. From 
\begin{equation}
    conv(\mathcal{P}_1, \mathcal{P}_2) = \{ \alpha \cdot p_1 + (1 - \alpha) \cdot p_2 \textrm{ }|\textrm{ } p_1 \in \mathcal{P}_1, \textrm{ } p_2 \in \mathcal{P}_2, \textrm{ } \alpha \in [0, 1]\}
\end{equation}
we know that $h(n) = h(n_1) + 2h(n_2) + 1$ with $n = n_1 + n_2$ and $n_i$ is the number of vertices of the polytopes being merged in this step.\\
%Now let $n$ be the number of vertices of a polytope in M-representation. 
With $h(m) = m - 1$ for every $m \in [n]$ and the claim of induction, it is obvious that a polytope with $n + 1$ vertices has the least number of generators if we choose $n_1 = n$ and $n_2 = 1$. Then we obtain $h(n + 1) = n - 1 + 2 \cdot 0 + 1 = n$, which proves the second part of the theorem. For the first part, we need the representation of a polytope $\mathcal{P}_{n}$ with $n$ vertices in order to compute the induction step. For this polytope we use 
\begin{equation} \label{sreppolyform}
    \mathcal{P}_{n} = \Big\{v_n + \sum_{i=1}^{n-1}\Big(\prod_{j=1}^{n-i} \alpha_{n-j}\Big) (v_i - v_{i+1})\Big\}_{\alpha}.
\end{equation}
For the representation of a polytope $\mathcal{P}_{n+1}$ with $n+1$ vertices we obtain
\begin{equation}
    \begin{split}
        \mathcal{P}_{n+1} & = conv(\mathcal{P}_{n}, v_{n+1}) =\\
        \textrm{ } & = \Big\{\alpha_n \Big(v_n + \sum_{i=1}^{n-1}\Big(\prod_{j=1}^{n-i} \alpha_{n-j}\Big) (v_i - v_{i+1})\Big) + (1 - \alpha_n) v_{n+1}\Big\}_{\alpha}=\\
        \textrm{ } & = \Big\{v_{n+1} + \sum_{i=1}^{n-1}\Big(\prod_{j=0}^{n-i} \alpha_{n-j}\Big) (v_i - v_{i+1}) + \alpha_n (v_n - v_{n+1})\Big\}_{\alpha}=\\
        \textrm{ } & = \Big\{v_{n+1} + \sum_{i=1}^{n}\Big(\prod_{j=0}^{n-i} \alpha_{n-j}\Big) (v_i - v_{i+1})\Big\}_{\alpha}=\\
        \textrm{ } & = \Big\{v_{n+1} + \sum_{i=1}^{n}\Big(\prod_{k=1}^{n+1-i} \alpha_{n+1-k}\Big) (v_i - v_{i+1})\Big\}_{\alpha}=\\
        \textrm{ } & = \langle v_{n+1}, \textrm{ } [v_1 - v_2, v_2 - v_3, \dots, v_n - v_{n+1}], \textrm{ } L_n\rangle_M. %\mathcal{E} \rangle_M,
    \end{split}
\end{equation}
%where $\mathcal{E} \in \mathbb{R}^{n \times n}$ is a lower triangular matrix filled with ones. 
In order to obtain this representation, $nd$ subtractions are necessary, therefore the complexity is in $\mathcal{O}(nd)$. For the representation size we only need to save the matrix of basis vectors and the index of the lower triangular matrix which leads to a complexity of $\mathcal{O}(nd)$. This proves the first part of the theorem.
\end{proof}

%REMOVE "MINIMALITY"? 
%Now we still need to prove that our strategy of adding the vertices one by one to the representation generates an M-representation with the minimal possible number of basis vectors for general polytopes, i.e. that the form shown in the above theorem results in the least possible number of basis vectors. For this, we need to have a closer look at the structure of the representation:\\

%Let $\alpha = (\alpha_1, \dots, \alpha_{n-1})$. Then from equation \eqref{sreppolyform} we can conclude the following: Let $\alpha$ and $\alpha'$ represent two points on a polytope and let $k$ be maximal s.t. $\alpha_k = 0$. Then $\alpha = \alpha'$ iff $\alpha_j = \alpha'_j$ for all $k \leq j \leq n-1$.

Let us have a look at why Theorem \ref{thmsrep} cannot be adapted for the Z-representation:

\begin{corollary}
A Z-representation of a polytope of the following form is always a point symmetric polytope:

\begin{equation} \label{zrepformzono}
    \mathcal{P} = \langle c, [v_1, v_2, \dots, v_{n-1}], L_{n-1} \rangle_Z. %\mathcal{E} \rangle_Z,
\end{equation}

%where $\mathcal{E} \in \mathbb{R}^{n-1 \times n-1}$ is a lower triangular matrix filled with ones.
\end{corollary}

\begin{proof}
Let $\mathcal{P}$ be defined as above. Then
\begin{equation} \label{zrepzonolikesrep}
    \mathcal{P} = \Big\{c + \sum_{i=1}^{n-1}\Big(\prod_{j=1}^{n-i} \alpha_{n-j}\Big) v_i \textrm{ } \Big| \textrm{ } \alpha \in [-1, 1]^{n-1} \Big\}.
\end{equation}
From the Z-representation we know that the vertices are the points of this set with $\alpha_i \in \{-1, 1\}$. 
Let $c + w$ be a vertex of $\mathcal{P}$ with $\alpha_{n-1} = 1$. Then by \eqref{zrepzonolikesrep} follows that $c - w$ is also a vertex of $\mathcal{P}$ if we only replace $\alpha_{n-1} = -1$. Therefore, each vertex has a partner which is point symmetric to the center $c$.
\end{proof}

From this we know that a general polytope with $n$ vertices cannot be represented by a Z-representation of the form in \eqref{zrepformzono} with $n-1$ generators. In Proposition \ref{srepconv} we prove another advantage of the M-representation over the Z-representation even if there exist representations of polytopes with the same number of generators/basis vectors.

% Let us now assume there exist Z-representations with $n-1$ generators of two polytopes with $n$ vertices each. Let us further assume that for the same polytopes no M-representations with less than $n-1$ basis vectors exist. Let $h_Z$ and $h_M$ denote the number of generators/basis vectors of the convex hulls of the respective representations. Then the convex hull of the M-representations of these polytopes has less basis vectors than the Z-representation has generators:

% \begin{equation}
% \begin{split}
%     h_Z &= 2 \cdot (n-1) + 2 \cdot (n-1) + 1 = 4n - 3\\
%     h_M &= (n-1) + 2 \cdot (n-1) + 1 = 3n - 2.
% \end{split}
% \end{equation}

% This means that even if there exists an equally good Z-representation for a polytope in terms of the number of generators, the M-representation still has advantages over the Z-representation.\\
The next corollary follows directly from the theorem above as well:

\begin{corollary} \label{corvertices}
Let $\mathcal{P}$ be a polytope of the form introduced in Theorem \ref{thmsrep}. Let $\alpha$ and $\alpha'$ represent two points in a polytope and let $k$ be maximal s.t. $\alpha_k = 0$. In case there is no such $k$, set $k=1$. Then $\alpha$ and $\alpha'$ represent the same point iff $\alpha_j = \alpha'_j$ for all $k \leq j \leq n-1$.
\end{corollary}

\begin{example}
    This means that for a polytope in the form of Theorem \ref{thmsrep} with three basis vectors, all $\alpha$'s of the form $\alpha = \begin{pmatrix}
        \alpha_1\\
        0\\
        1
        \end{pmatrix}$ represent the same point in the polytope.
\end{example}

For a polytope $\mathcal{P}$ in the form of Theorem \ref{thmsrep}, the vertices can be computed by all combinations of the $\alpha_i \in \{0, 1\}$.
With Corollary \ref{corvertices} the vertices of $\mathcal{P}$ can be represented by the $\alpha$'s of the form of the columns of $L_{n-1}$ and the zero vector. 
%Hence, we know that each basis vector introduces an additional vertex on top of the starting point which is a vertex as well. 
From Theorem \ref{thmsrep} it is clear that we can recover the vertices iteratively in $\mathcal{O}(nd)$ operations.
%In order to get the vertices from this representation, we can do the following:

%If the polytope is represented in the form introduced in theorem \ref{thmsrep}, we can do the following: 
%The first vertex that we get is $s$, the starting point. This is represented by $\alpha = \mathbf{0}$. The second vertex can be obtained by adding the last column in the basis vector matrix. This can be done in $\mathcal{O}(d)$ operations, since we only need to add one value in each entry of the first vertex. Next, we repeat the same step with the second to last column of the basis vector matrix. If we continue iteratively, we get all $n$ vertices of the polytope in $\mathcal{O}(nd)$ operations.
%Another advantage of this form is that we can directly read the number of vertices from this M-representation without further computations. 

%This results in the following proposition:

%\begin{proposition}
%The M-representation proven in theorem \ref{thmsrep} has the minimal possible number of basis vectors for general polytopes in $\mathbb{R}^d$. REMOVE?
%\end{proposition}

\begin{remark}
If we define the starting point of the M-representation as a basis vector as well, this representation has the same number of basis vectors for general polytopes as the V-representation has vertices. 
%But in comparison to the V-representation, we still need the exponent matrix for this representation. Sec. \ref{repsize} deals with this issue.
\end{remark}

\begin{proposition}
    Obtaining a chain form of a general polytope $\mathcal{P}$ in M-representation with $h$ basis vectors and $p$ factors can be done in $\mathcal{O}(2^phd)$.
\end{proposition}

\begin{proof}
    In order to obtain a chain form of a polytope $\mathcal{P}$ in M-representation, we need to compute the $2^p$ potential vertices of $\mathcal{P}$. In general this can be done in $\mathcal{O}(2^phd)$. From Theorem \ref{thmsrep} we know that we can obtain the chain form of these vertices in $\mathcal{O}(2^pd)$.
\end{proof}

\section{Operations on Polytopes in M-Representation} \label{operationssrep}

In this section we present the linear transformation, Minkowski sum, and convex hull of polytopes in M-representation. 

\begin{theorem} \label{sreptranmink}
The M-representation of a polytope $\mathcal{P} = \langle s, B, \mathcal{E} \rangle_M$ in $\mathbb{R}^d$ with $h$ basis vectors directly inherits the efficient computation of linear transformations and Minkowski sums from the Z-representation. A linear transformation by $M \in \mathbb{R}^{m \times d}$ can be computed as 
\begin{equation}
    M \mathcal{P} = \langle Ms, MB, \mathcal{E} \rangle_M.
\end{equation}
This can be done in $\mathcal{O}(mdh)$.\\
The Minkowski sum of two polytopes $\mathcal{P}_1 = \langle s_1, B_1, \mathcal{E}_1 \rangle_M$ and $\mathcal{P}_2 = \langle s_2, B_2, \mathcal{E}_2 \rangle_M$ can be computed as
\begin{equation} \label{srepminkthm}
    \mathcal{P}_1 \oplus \mathcal{P}_2 = \Big\langle s_1+s_2, [B_1, B_2],
        \begin{pmatrix}
        \mathcal{E}_1 & \mathbb{O}\\
        \mathbb{O} & \mathcal{E}_2
        \end{pmatrix} \Big\rangle_M.
\end{equation}
This can be computed in $\mathcal{O}(d)$ and the representation size is in $\mathcal{O}\Big((h_1+h_2) \cdot max\{h_1+h_2, d\}\Big)$.\\
If $\mathcal{P}_1$ and $\mathcal{P}_2$ are in chain form, the matrix of exponents can be represented as $\begin{pmatrix}
        L_{h_1} & \mathbb{O}\\
        \mathbb{O} & L_{h_2}
        \end{pmatrix}$.
This has a representation size in $\mathcal{O}\Big((h_1+h_2) d\Big)$.
\end{theorem}

\begin{proof}
The proof is identical to the one for the Z-representation in \cite{Kochdumper} since all operations are independent of the range of the factors $\alpha_i$. For the representation of polytopes in chain form we only need to save the matrix of basis vectors and a $2 \times 2$ matrix filled with the indices of the lower triangular matrices and the $\mathbb{O}$ symbol. Therefore the representation size is in $\mathcal{O}\Big((h_1+h_2) d\Big)$. 
\end{proof}

In \cite{Gritzmann} the number of vertices for the Minkowski sum of $k$ polytopes in $\mathbb{R}^d$ is discussed. If each of the polytopes has at most $n$ vertices, the total number of vertices is in $\mathcal{O}(k^{d-1}n^{2(d-1)})$.\\
From Theorem \ref{sreptranmink} follows that the number of basis vectors of $k$ polytopes is the sum of the number of basis vectors of each polytope. Let each polytope be represented by at most $n$ basis vectors, i.e. the polytope has at least $n$ vertices if we use the representation from Theorem \ref{thmsrep}. Then the number of basis vectors of the Minkowski sum of these polytopes is in $\mathcal{O}(kn)$.\\

% From Proposition 2.3.5 in \cite{Gritzmann} follows that we can get the Minkowski sum of $k$ polytopes in $\mathbb{R}^d$ with at most $n$ vertices in $\mathcal{O}\Big(n^{k \cdot (2-\frac{2}{1+\lfloor \frac{d}{2} \rfloor} + \delta)}\Big)$ arithmetic operations, where $\delta$ is a fixed positive real value. 
% By using the M-representation we can improve this to $\mathcal{O}(kdn)$:

% \begin{corollary}
% The Minkowski sum of $k$ polytopes in $\mathbb{R}^d$ in V-representation with at most $n$ vertices each can be obtained in $\mathcal{O}(kdn)$.
% \end{corollary}

% \begin{proof}
% Computing M-representations of the polytopes as described in Theorem \ref{thmsrep} from the V-representations can be done in $\mathcal{O}(kdn)$. As mentioned in Theorem \ref{sreptranmink}, the Minkowski sum of two polytopes in M-representation can be obtained in $\mathcal{O}(d)$. The Minkowski sum of $k$ polytopes can therefore be obtained in $\mathcal{O}(kd)$. As a last step, we need to get the vertices again, which can be done in $\mathcal{O}(kdn)$.
% \end{proof}

Let $\mathcal{P}_{M1}$ and $\mathcal{P}_{M2}$ be two polytopes in M-representation with $h_M$ basis vectors each, $\mathcal{P}_{Z1}$ and $\mathcal{P}_{Z2}$ be two polytopes in Z-representation with $h_Z$ generators each and $h_Z = h_M$. Then $conv(\mathcal{P}_{M1}, \mathcal{P}_{M2})$ needs $h_M$ basis vectors less than $conv(\mathcal{P}_{Z1}, \mathcal{P}_{Z2})$ needs generators:

%Let $h_Z = h_S$ be the number of generators/basis vectors of a polytope in Z-/M-representation, respectively. Then the convex hull of two polytopes in M-representation needs $h_1$ basis vectors less than the convex hull of the Z-representation of these polytopes needs generators:

\begin{proposition} \label{srepconv}
Let $\mathcal{P}_1 = \langle s_1, B_1, \mathcal{E}_1 \rangle_M$ and $\mathcal{P}_2 = \langle s_2, B_2, \mathcal{E}_2 \rangle_M$ be two polytopes in M-representation with $h_1 \geq h_2$ being the respective number of basis vectors. Then

\begin{equation}
    \begin{split}
        \mathcal{P} & = conv(\mathcal{P}_1, \mathcal{P}_2) = \langle s_2, [B_2, -B_2, B_1, s_1 - s_2], \mathcal{E} \rangle_M,
    \end{split}
\end{equation}

with

\begin{equation}
    \mathcal{E} = %(\mathcal{E}_2, \begin{pmatrix}
        %\mathcal{E}_1\\
        %\alpha_{h_1 + h_2 + 1} \cdot \mathcal{I}_{(1, h_1)}
        %\end{pmatrix}, \begin{pmatrix}
        %\mathcal{E}_1\\
        %p \cdot \mathcal{I}_{(1, h_1)}
        %\end{pmatrix} blablabla)
        \begin{pmatrix}
        \mathcal{E}_2 & \mathcal{E}_2 & \mathbb{O} & \mathbb{O}\\
        \mathbb{O} & \mathbb{O} & \mathcal{E}_1 & \mathbb{O}\\
        %\textbf{0} & \mathcal{I}_{(1, h_2)} & \mathcal{I}_{(1, h_1)} & \mathcal{I}_{(1, 1)}
        \mathbb{O} & \mathbb{I} & \mathbb{I} & 1
        \end{pmatrix}
\end{equation}

being a block matrix and $\mathcal{P}$ has $h = h_1 + 2h_2 + 1$ basis vectors.
The complexity of obtaining the convex hull of two polytopes in M-representation is in $\mathcal{O}(d)$ and the representation size is in $\mathcal{O}\Big(max\Big\{(h_1+2h_2+1)d, (h_1+h_2+1)^2\Big\}\Big)$.\\
If $\mathcal{P}_1$ and $\mathcal{P}_2$ are in chain form, the matrix of exponents can be represented as 

\begin{equation}
    \mathcal{E} = \begin{pmatrix}
        L_{h_2} & L_{h_2} & \mathbb{O} & \mathbb{O}\\
        \mathbb{O} & \mathbb{O} & L_{h_1} & \mathbb{O}\\
        \mathbb{O} & \mathbb{I} & \mathbb{I} & 1
        \end{pmatrix}.
\end{equation}

This has a representation size in $\mathcal{O}\Big((h_1+2h_2+1)d\Big)$.
\end{proposition}

\begin{proof}
We write two polytopes $\mathcal{P}_1 = \langle v_n, B_1, \mathcal{E}_1 \rangle_M$ and $\mathcal{P}_2 = \langle w_m, B_2, \mathcal{E}_2 \rangle_M$ as $\mathcal{P}_1 = \Big\{v_n + \sum_{i=1}^{n-1}\Big(\prod_{j=1}^{n-i} \alpha_{n-j}\Big) (v_i - v_{i+1})\Big\}_{\alpha}$ and $\mathcal{P}_2 = \Big\{w_m + \sum_{i=1}^{m-1}\Big(\prod_{j=1}^{m-i} \alpha_{m-j}\Big) (w_i - w_{i+1})\Big\}_{\alpha}$ with $n-1$ and $m-1$ basis vectors, respectively. Then

\begin{equation}
    \begin{split}
        \mathcal{P} & = conv(\mathcal{P}_1, \mathcal{P}_2) =\\
        \textrm{ } & = \Big\{\alpha_{n+m-1} \Big(v_n + \sum_{i=1}^{n-1}\Big(\prod_{j=1}^{n-i} \alpha_{n-j}\Big) (v_i - v_{i+1})\Big) +\\
        \textrm{ } & + (1 - \alpha_{n+m-1}) \Big(w_m + \sum_{i=1}^{m-1}\Big(\prod_{j=1}^{m-i} \alpha_{m-j}\Big) (w_i - w_{i+1})\Big)\Big\}_{\alpha}=\\
        \textrm{ } & = \Big\{w_m + \sum_{i=1}^{m-1}\Big(\prod_{j=1}^{m-i} \alpha_{m-j}\Big) (w_i - w_{i+1}) + \alpha_{n+m-1} \Big(-\sum_{i=1}^{m-1}\Big(\prod_{j=1}^{m-i} \alpha_{m-j}\Big) (w_i - w_{i+1})\Big) +\\
        \textrm{ } & + \alpha_{n+m-1} \Big(\sum_{i=1}^{n-1}\Big(\prod_{j=1}^{n-i} \alpha_{n-j}\Big) (v_i - v_{i+1})\Big) + \alpha_{n+m-1} (v_n - w_m)\Big\}_{\alpha}=\\
        \textrm{ } & = \langle w_m, [G_2, -G_2, G_1, v_n - w_m], \mathcal{E} \rangle_M
    \end{split}
\end{equation}
with $\mathcal{E}$ as defined above.\\
For the representation of polytopes in chain form we only need to save the matrix of basis vectors and a $3 \times 4$ matrix filled with the indices of the lower triangular matrices, the symbols $\mathbb{O}$ and $\mathbb{I}$ and a 1. Therefore the representation size is in $\mathcal{O}\Big((h_1+2h_2+1)d\Big)$.
\end{proof}

A more compact representation of $conv(\mathcal{P}_1, \mathcal{P}_2)$ in terms of number of basis vectors can be obtained by computing the vertices of $\mathcal{P}_1$ and $\mathcal{P}_2$, deleting the ones that are not vertices of the convex hull of $\mathcal{P}_1$ and $\mathcal{P}_2$, and using the strategy from Theorem \ref{thmsrep} to obtain the M-representation of the remaining vertices. This method would result in a maximum of $h = n_1 + n_2 - 1 = h_1 + h_2 + 1$ basis vectors. However, the computational complexity of deleting the vertices that are not vertices of the convex hull has exponential complexity in the number $d$ of dimensions, similar to the computation of the convex hull of two polytopes in V-representation \cite{Tiwary}. In order to reduce the representation size of the  convex hull for polytopes in chain form, we introduce a variant of the M-representation in Sec. \ref{repsize}.

\section{Algorithm for Reducing the Number of Basis Vectors in M-Representation} \label{algreducedsrep}

Now we want to introduce an algorithm that returns an M-representation with at most $n-1$ basis vectors for a polytope with $n$ vertices. In case the vertices of the polytope fulfill certain condidtions, this algorithm returns less than $n-1$ basis vectors. Let us have a closer look at zonotopes in M-representation for this. From \cite{Gritzmann} we know that zonotopes are Minkowski sums of line segments. We can use this for the M-representation of zonotopes.

% \begin{proposition} \label{srepzonomethod}
% Let $\mathcal{Z}$ be an $m$-dimensional zonotope in $\mathbb{R}^d$ with $n$ vertices which is spanned by the Minkowski sum of $h$ line segments. Then we can express $\mathcal{Z}$ by an M-representation with $h$ basis vectors. The only difference compared to the Z-representation is that the starting point is one of the vertices and the generators are double the length.
% \end{proposition}

% \begin{proof}
% Each of the line segments can be expressed as an M-representation with one basis vector. Using Theorem \ref{sreptranmink} proves the number of basis vectors. That the basis vectors are double the length of the generators from the Z-representation and that the starting point is a vertex follows directly from the expression of a line segment by an M-representation.
% \end{proof}

\begin{proposition} \label{srepzonomethod}
Let $\mathcal{Z}$ be an $m$-dimensional zonotope in $\mathbb{R}^d$ with $n$ vertices which is spanned by the Minkowski sum of $h$ line segments. Let the line segments be of the form $[l_{i1}, l_{i2}]$ with $i \in [h]$, where $l_{i1}$ is the starting point and $l_{i2}$ is the end point of this line segment. Then we can express $\mathcal{Z}$ as
\begin{equation*}
    \mathcal{Z} = \Big\langle \sum_{i=1}^h l_{i1}, [l_{12} - l_{11}, l_{22} - l_{21}, \dots, l_{h2} - l_{h1}], I_h \Big\rangle_M.
\end{equation*}
This representation uses $h$ basis vectors.
\end{proposition}

\begin{proof}
    We can represent each line segment $l_i$ by 
    \begin{equation}
        l_i = \langle l_{i1}, [l_{i2} - l_{i1}], 1 \rangle_M.
    \end{equation}
    By applying Theorem \ref{sreptranmink} $h-1$ times, we obtain the stated result with $h$ basis vectors.
\end{proof}

%From the Z-representation we inherit the property of representing zonotopes with less than $n$ generators. In fact, we get the same number of generators for an M-representation as in Z-representation with the only difference that the center point is one of the vertices and the generators have twice the length of the generators in Z-representation.\\

\begin{lemma} \label{lemmasrepzono}
Let $\mathcal{Z}$ be an $m$-dimensional zonotope in $\mathbb{R}^d$ with $n$ vertices and $m, h \neq 0$. Then we can represent $\mathcal{Z}$ by at most $h\leq \frac{n}{2}$ pairwise distinct basis vectors in M-representation.
\end{lemma}

\begin{proof}
For the cases with $m < 2$ and $h < 2$ this is clear. For all other cases we can use Proposition 2.1.2 in \cite{Gritzmann}: For an $m$-dimensional zonotope in $\mathbb{R}^d$ with $h$ pairwise distinct basis vectors in M-representation and $n$ vertices the following relation holds:

\begin{equation}\label{srepzonorep}
    n=2 \sum_{i=0}^{min\{m,h\}-1} \binom{h-1}{i}
\end{equation}

From this, we obtain the following inequality:

\begin{equation}
    n=2 \sum_{i=0}^{min\{m,h\}-1} \binom{h-1}{i} \geq 2 \binom{h-1}{0} + 2\binom{h-1}{1} = 2h
\end{equation}
\end{proof}

Hence, we can represent every zonotope by less than $n-1$ basis vectors for $n > 2$. We can use this to represent general polytopes, where a subset of the vertices forms a zonotope, by less than $n-1$ basis vectors. Alg. \ref{srepstrat} returns at most $n-1$ basis vectors for general polytopes with $n$ vertices.\\
% \begin{algorithm}[H]
% \caption{Algorithm for obtaining an M-representation} \label{srepstrat}
% \textit{Input}: Set of $n$ vertices in $\mathbb{R}^d$\\
% \textit{Output}: an M-representation of the polytope spanned by these $n$ vertices with at most $h = n-1$ basis vectors as Theorem \ref{thmsrep} produces.
% \begin{itemize}
%     \item[1.] Check whether a subset of vertices containing more than 2 vertices forms a zonotope.
%     \item[2.] If this is not the case, use Theorem \ref{thmsrep}.
%     \item[3.] If this is the case, use Proposition \ref{srepzonomethod} for the representation of this subset and use Proposition \ref{srepconv} to add the remaining vertices one by one to the zonotope.
% \end{itemize}
% % \begin{algorithmic}
% % \if{$\exists S \subseteq V$ s.t. $S$ builds a zonotope and $S$ maximal}
% % \state use Proposition \ref{srepzonomethod} on $S$ and use Proposition \ref{srepconv} to add the vertices in $V \setminus S$ one by one to the representation of $S$.
% % \else
% % \state use Theorem \ref{thmsrep} on the vertex set $V$.
% % \endif
% % \end{algorithmic}
% \end{algorithm}

%\SetKwComment{Comment}{/* }{ */}

\begin{algorithm}[H]
\caption{Algorithm for obtaining an M-representation} \label{srepstrat}
\textbf{Input:} Set $V$ containing $n$ vertices in $\mathbb{R}^d$\\
\textbf{Output:} M-representation of the polytope spanned by $V$ with at most $h = n-1$ basis vectors\\
  \eIf{$\exists S \subseteq V$ s.t. $S$ spans a zonotope, $|S|>2$ and $S$ maximal}{
    use Proposition \ref{srepzonomethod} on $S$ and use Proposition \ref{srepconv} to add the vertices in $V \setminus S$ one by one to the representation of $S$
  }{
      use Theorem \ref{thmsrep} on the vertex set $V$
}
\end{algorithm}
\textrm{ }\\
It is clear that this representation is multilinear again.\\
For checking whether a set is a zonotope we can use \cite[Alg. 3]{Deza}, which introduces an algorithm for checking whether a set of vertices forms a zonotope. This algorithm also returns the line segments which span the zonotope by their Minkowski sum. In order to be able to represent this polytope by an M-representation, we still need to apply Proposition \ref{srepzonomethod}. As candidates for such a set $S$ as described in Alg. \ref{srepstrat}, we only need to take sets into account that are point symmetric to a center as this is a necessary criterion for a set to be a zonotope.\\
Now we look at an example of an application of Alg. \ref{srepstrat} that reduces the number of basis vectors $h$ from 4 to 3 for 5 vertices.

\begin{example}
Let $\mathcal{P}$ be a polytope with the 5 vertices $\begin{pmatrix}
        0\\
        0
        \end{pmatrix}, \begin{pmatrix}
        0\\
        2
        \end{pmatrix}, \begin{pmatrix}
        2\\
        2
        \end{pmatrix}, \begin{pmatrix}
        2\\
        0
        \end{pmatrix}, \begin{pmatrix}
        1\\
        3
        \end{pmatrix}$. 
The first four vertices form a zonotope that can be represented by 
\begin{equation}
    \mathcal{P}' = \Big\{\begin{pmatrix}
        0\\
        0
        \end{pmatrix} + \alpha_1 \cdot \begin{pmatrix}
        0\\
        2
        \end{pmatrix} + \alpha_2 \cdot \begin{pmatrix}
        2\\
        0
        \end{pmatrix}\Big\}_{\alpha}.
\end{equation}
$\mathcal{P}$ can be written in M-representation as
\begin{equation}
\begin{split}
    \mathcal{P} & = \Big\{\alpha_3 \cdot \mathcal{P}' + (1 - \alpha_3) \cdot \begin{pmatrix}
        1\\
        3
        \end{pmatrix}\Big\}_{\alpha}\\
        & = \Big\{\begin{pmatrix}
        1\\
        3
        \end{pmatrix} + \alpha_3 \cdot \begin{pmatrix}
        -1\\
        -3
        \end{pmatrix} + \alpha_1 \alpha_3 \cdot \begin{pmatrix}
        0\\
        2
        \end{pmatrix} + \alpha_2 \alpha_3 \cdot \begin{pmatrix}
        2\\
        0
        \end{pmatrix}\Big\}_{\alpha},
\end{split}
\end{equation}

which has only 3 basis vectors.
\end{example}

\begin{example} \label{examplezono}
Let us have a look at the polytope shown in Fig. \ref{exzono}. In M-representation, this polytope could be represented by

\begin{equation}
    \begin{split}
        \mathcal{P} & = \Big\langle \begin{pmatrix}
        -2\\
        -1
        \end{pmatrix}, \Big[\begin{pmatrix}
        -2\\
        -2
        \end{pmatrix}, \begin{pmatrix}
        2\\
        0
        \end{pmatrix}, \begin{pmatrix}
        2\\
        2
        \end{pmatrix}\Big],
        \begin{pmatrix}
        1 0 0\\
        1 1 0\\
        1 1 1
        \end{pmatrix}
        \Big\rangle_M =\\
        \textrm{ } & = \Big\{\begin{pmatrix}
        -2\\
        -1
        \end{pmatrix} + \alpha_1 \alpha_2 \alpha_3 \cdot \begin{pmatrix}
        -2\\
        -2
        \end{pmatrix} + \alpha_2 \alpha_3 \cdot \begin{pmatrix}
        2\\
        0
        \end{pmatrix} + \alpha_3 \cdot \begin{pmatrix}
        2\\
        2
        \end{pmatrix}\Big\}_{\alpha}.
    \end{split}
\end{equation}

If we apply Alg. \ref{srepstrat}, we can even express it as

\begin{equation}
    \begin{split}
        \mathcal{P} & = \Big\langle \begin{pmatrix}
        -2\\
        -1
        \end{pmatrix}, \Big[\begin{pmatrix}
        2\\
        0
        \end{pmatrix}, \begin{pmatrix}
        2\\
        2
        \end{pmatrix}\Big],
        \begin{pmatrix}
        1 0\\
        0 1
        \end{pmatrix}
        \Big\rangle_M =\\
        \textrm{ } & = \Big\{\begin{pmatrix}
        -2\\
        -1
        \end{pmatrix} + \alpha_1 \cdot \begin{pmatrix}
        2\\
        0
        \end{pmatrix} + \alpha_2 \cdot \begin{pmatrix}
        2\\
        2
        \end{pmatrix}\Big\}_{\alpha}.
    \end{split}
\end{equation}
\end{example}

\section{Chain Representation of the Chain Form} \label{repsize}

%Now we are going to introduce two variants of the M-representation, which make it possible to reduce the representation size of polytopes with $h$ basis vectors in $\mathbb{R}^d$ to $\mathcal{O}(hd)$. The first one is only suited for convex hulls and the second one can also handle Minkowski sums and convex hulls.
Now we introduce a variant of the M-representation, which makes it possible to reduce the computational complexity of the convex hull of two polytopes in $\mathbb{R}^d$ with $h$ basis vectors each and in chain form to $\mathcal{O}(d)$:

%\subsection{Representation of a Polytope and Convex Hulls}

%Let us refer to the M-representation introduced in Theorem \ref{thmsrep} as the \textit{chain form}. In order to reduce the representation size, we can neglect the exponent matrix of the chain form. 
%Let us define this notation:

\begin{definition} \label{thmcrepforconv}
    Let $s \in \mathbb{R}^d$ be a starting point, $B \in \mathbb{R}^{d \times h}$ a matrix of basis vectors, $\mathcal{E} = L_h$ a matrix of exponents and $e \in \mathbb{R}^d$ an end point, then a \textit{chain representation (C-representation)} of a polytope $\mathcal{P} \subseteq \mathbb{R}^d$ in chain form is defined as
    
    \begin{equation}
        \mathcal{P} = \Big\{s + \sum_{i=1}^h(\prod_{k=1}^p\alpha_k^{\mathcal{E}_{(k,i)}})B_{(\cdot,i)} \textrm{ } \Big| \textrm{ } \alpha \in [0, 1]^p \Big\}
    \end{equation}
    
    and we write
    
    \begin{equation}
        \mathcal{P} = \langle s, B, e \rangle_C.
    \end{equation}
\end{definition}

%This modification makes it possible to reduce the representation size from $\mathcal{O}(h \cdot max\{d, h\})$ to $\mathcal{O}(hd)$. Saving the end point helps us for the next proposition:
The basis vectors appearing in $B$ connect the starting point and the end point, which looks like a chain.
It is sufficient so save $s$, $B$ and $e$, since the exponent matrix in chain form is uniquely defined by the dimensions of $B$.
This variant has a representation size in $\mathcal{O}(hd)$. Saving the end point helps us for the next proposition:

\begin{proposition} \label{crepforconv}
Let $\mathcal{P}_1 = \langle s_1, B_1, e_1 \rangle_C$ and $\mathcal{P}_2 = \langle s_2, B_2, e_2 \rangle_C$ be two polytopes in $\mathbb{R}^d$ in C-representation with $h_1$ and $h_2$ being the respective number of basis vectors. Then

\begin{equation} \label{convchain}
    \begin{split}
        \mathcal{P} & = conv(\mathcal{P}_1, \mathcal{P}_2) = \langle s_1, [B_1, s_2 - e_1, B_2], e_2 \rangle_C
    \end{split}
\end{equation}

and $\mathcal{P}$ has $h = h_1 + h_2 + 1$ basis vectors.
The complexity of obtaining the convex hull of two polytopes in C-representation is $\mathcal{O}(d)$ and can be represented in $\mathcal{O}(hd)$.
\end{proposition}

\begin{proof}
The set of vertices of the convex hull of $\mathcal{P}_1$ and $\mathcal{P}_2$ is a subset of the union of the vertices of these polytopes. It is clear from the definition of the chain form, that \eqref{convchain} represents a polytope. From Theorem \ref{thmsrep} also follows that the vertices of the polytope in \eqref{convchain} can be represented by $\alpha$'s of the form of the zero vector and the columns of a lower triangular matrix filled with ones with dimensions $(h_1+h_2+1) \times (h_1+h_2+1)$. The matrix of basis vectors is the connection of the original chains $B_1$ and $B_2$ by the link between the end point $e_1$ and the starting point $s_2$. This returns as vertices of the polytope $\mathcal{P}$ the union of the vertex sets of $\mathcal{P}_1$ and $\mathcal{P}_2$. Hence, \eqref{convchain} represents the required convex hull.
\end{proof}

\begin{proposition}
    The C-representation of a polytope $\mathcal{P} = \langle s, B, e \rangle_C$ in $\mathbb{R}^d$ with $h$ basis vectors directly inherits the efficient computation of linear transformations from the M-representation. A linear transformation by $M \in \mathbb{R}^{m \times d}$ can be computed as 
\begin{equation}
    M \mathcal{P} = \langle Ms, MB, Me \rangle_C.
\end{equation}
This can be done in $\mathcal{O}(mdh)$.
\end{proposition}

\begin{proof}
The proof is identical to the one for the Z-representation in \cite{Kochdumper}. 
\end{proof}

\end{document}